\documentclass{gtpart}     
\usepackage{graphicx}
\usepackage{hyperref}
\usepackage[curve]{xypic}

\title{Spectral sequences in string topology}

\author{Lennart Meier}
\givenname{Lennart}
\surname{Meier}
\address{Mathematisches Institut\\
Endenicher Allee 60\\
53115 Bonn\\
Germany}
\email{lmeier@math.uni-bonn.de}
\urladdr{http://www.math.uni-bonn.de/people/lmeier/}


\arxivreference{math.AT/10014906}
\arxivpassword{widxx}

\volumenumber{}
\issuenumber{}
\publicationyear{}
\papernumber{}
\startpage{}
\endpage{}
\doi{}
\MR{}
\Zbl{}
\received{}
\revised{}
\accepted{}
\published{}
\publishedonline{}
\proposed{}
\seconded{}
\corresponding{}
\editor{}
\version{}

\newtheorem{thm}{Theorem}[section]    
\newtheorem{lemma}[thm]{Lemma}

\newtheorem{prop}[thm]{Proposition}
\theoremstyle{definition}
\newtheorem{defi}[thm]{Definition}    

\newtheorem*{remark}{Remark}             
\newtheorem*{warning}{Warning}

\newcommand{\nbhd}{neighbourhood }

\newcommand{\tensor}{\otimes}

\newcommand{\N}{\mathbb{N}}
\newcommand{\h}{\textbf{h}}
\newcommand{\HH}{\mathbb{H}}
\newcommand{\mm}{L^n M\times_M L^n M}

\newcommand{\cp}{\mathbb{CP}}
\newcommand{\hp}{\mathbb{HP}}
\newcommand{\kp}{\mathbb{KP}}
\newcommand{\for}{\text{ for }}

\newcommand{\EE}{\mathcal{E}}

\DeclareMathOperator{\ev}{ev}

\DeclareMathOperator{\id}{id}
\DeclareMathOperator{\im}{im}
\DeclareMathOperator{\pr}{pr}

\DeclareMathOperator{\Tor}{Tor}
\DeclareMathOperator{\colim}{colim}

\begin{document}

\begin{abstract}    
In this paper, we investigate the behavior of the Serre spectral sequence with respect to the algebraic structures of string topology in generalized homology theories, specificially with the Chas--Sullivan product and the corresponding coproduct and module structures. We prove compatibility for two kinds of fiber bundles: the fiber bundle $\Omega^n M\to L^n M\to M$ for an $h_*$--oriented manifold $M$ and the looped fiber bundle $L^n F\to L^n E\to L^n B$ of a fiber bundle $F\to E\to B$ of $h_*$--oriented manifolds. Our method lies in the construction of Gysin morphisms of spectral sequences. We apply these results to study the ordinary homology of the free loop spaces of sphere bundles and some generalized homologies of the free loop spaces of spheres and projective spaces. For the latter purpose, we construct explicit manifold generators for the homology of these spaces.
\end{abstract}

\maketitle


\section{Introduction}
Let $h_*$ be a homology theory and $M$ be a $d$--dimensional $h_*$--oriented smooth manifold. In \cite{chas1999string}, Chas and Sullivan defined a product on the singular homology of the free loop space and Cohen and Jones generalized it in \cite{cohen2002homotopy} to the case of an arbitrary homology theory $h_*$. This product is now called the \textit{Chas--Sullivan product} and is of the form
\[h_p(LM)\otimes h_q(LM) \to h_{p+q-d}(LM). \]
Later several people generalized this product, giving a whole bunch of algebraic structures on different mapping spaces between manifolds, for example coproduct and module structures. A description of some of these can be found in the exposition paper \textit{Notes on string topology} by Cohen and Voronov \cite{cohen2005notes}.

To enhance the calculational perspectives, it is, of course, useful to understand the behavior of these algebraic structures in the Serre spectral sequence associated to certain fiber bundles. The first theorem of this kind was proven in by Cohen, Jones and Yan in \cite{cohen2003loop}. We will generalize their theorem to arbitrary homology theories. Denoting the Serre spectral sequence of a fiber bundle $\xi$ by $\EE(\xi)$ and by $[(a,b)]$ a bidegree shift, we can state two of our main results:
{\renewcommand{\thethm}{\ref{cohen2003loop}}
\begin{thm}Let $M$ be a $d$--dimensional $h_*$--oriented manifold. Then 
\begin{eqnarray*}\EE(\Omega^n M \to L^n M\to M)[(d,0)]\end{eqnarray*}
 can be equipped with the structure of a multiplicative spectral sequence which converges to the Chas--Sullivan product on $h_*(L^n M)$. Furthermore, the induced product on the $E^2$--term $H_{*+d}(M; \mathfrak{h}_q(\Omega^n M))$ is equal to the intersection product with coefficients in the local system of rings $\mathfrak{h}_*(\Omega^n M)$ whose multiplication is given by the Pontryagin product.\end{thm}}
{\renewcommand{\thethm}{\ref{fss}}
\begin{thm}Let $M\to N\to O$ be a fiber bundle of $h_*$--oriented manifolds of dimensions $m$, $n$ and $o$ respectively, with projection map $\pi$. Then 
\begin{eqnarray*}\EE(L^n M\to L^n N \to L^n O)[(o,m)]\end{eqnarray*}
 can be equipped with the structure of a multiplicative spectral sequence which converges to the Chas--Sullivan product on $h_*(L^n N)$. Furthermore, the induced product on the $E^2$--term $H_{p+m}(L^n O; \mathfrak{h}_{q+m}(L^n M))$ is equal to the Chas--Sullivan product with coefficients in the local system of rings $\mathfrak{h}_{*+m}(L^n O)$.\end{thm}} 
There are similar theorems about (Goresky--Hingston) coproduct and module structures. 

We get these theorems as corollaries from the existence of two kinds of Gysin morphisms (i.e. ''wrong-way maps'') of spectral sequences, which may be of independent interest and lie at the technical heart of this paper. 
{\renewcommand{\thethm}{\ref{ib}}
\begin{thm}[Intersection on the base]\label{ib0} Let $\xi$ be a fiber bundle with a finite-dimensional manifold $B$ as base and fiber $F$ and let $A\subset B$ be a closed submanifold of codimension $d$ with $h_*$--oriented normal bundle. Then there is a morphism $s_B(A)$ of convergent spectral sequences of bidegree $(-d,0)$ between the Serre spectral sequences $\EE(\xi)$ and $\EE(\xi|_A)$ which induces the usual Gysin morphism $H_p(B; \mathfrak{h}_q(F))\to H_{p-d}(A; \mathfrak{h}_q(F))$ on $E^2$. Furthermore, it converges to the Gysin morphism in the homology of the total spaces.\end{thm}}
{\renewcommand{\thethm}{\ref{if1}}
\begin{thm}[Intersection in the fiber]Let $\xi = (F\to E\to B)$ be a smooth fiber bundle with base a finite-dimensional manifold and fiber a Hilbert manifold, $\xi' = (F'\to E'\xrightarrow{\pi'}B)$ an open subbundle, $\xi_0 = (F_0\to E_0\to B)$ be a subbundle of constant codimension $d$ with $h_*$--oriented normal bundle and $\xi'_0 = (F_0'\to E_0' \to B)$ be the intersection of the two subbundles. Then there is a morphism $s_F(E_0)$ of convergent spectral sequences of bidegree $(0,-d)$ between $\EE(\xi, \xi')$ and $\EE(\xi_0, \xi'_0)$. This induces the usual Gysin morphism $H_p(B; \mathfrak{h}_q(F, F'))\to H_p(B; \mathfrak{h}_{q-d}(F_0, F'_0)))$ on $E^2$. Furthermore, it converges to the Gysin morphism in the homology of the total spaces.\end{thm}}

The proofs of these results use Jakob's bordism-like description of an arbitrary homology theory, which was introduced into string topology by Chataur in \cite{chataur2005bordism}. Both results can be generalized to suitable Hilbert manifolds as base. 

As in every mathematical discipline, it is crucial to compute and understand examples to fill the abstract definitions with life. First, we want to compute the Chas--Sullivan product for a certain class of sphere bundles over spheres by rational homotopy theory. Then, we give explicit manifold generators for the singular homologies of the free loop spaces of spheres and projective spaces and use this informations to do computations in complex and oriented bordism and in Landweber exact theories. We get complete answers in the case of odd-dimensional spheres. At last,  we do a sample computation for the Goresky--Hingston coproduct. \\

Most of the theorems about the behavior of the Serre spectral sequence with respect to the Chas--Sullivan product and other algebraic structures were already shown by other people in the case of singular homology: We already mentioned \cite{cohen2003loop}. Le Borgne (\cite{le2008loop}) has constructed the Gysin morphisms of spectral sequences and applied them to something analogous to \ref{fss}. Kallel and Salvatore (\cite{kallel2003rational}) have proven compatibility of the Serre spectral sequences associated to $\Omega^n M\to L^n M\to M$ with module structures. To the knowledge of the author, all the results about spectral sequences are new for other homology theories. It should be noted that the techniques of the mentioned authors do not generalize since they use chain methods. Furthermore, they do not treat the compatibility with the Goresky--Hingston coproduct. 

Le Borgne has also computed homologies of free loop spaces of sphere bundles in some other cases than in this paper by a different method. \\

The paper is structured as follows:

In section 2, we discuss some preliminaries. First, we recall the definition of a Hilbert manifold and some of their properties. These notions are important for our project since (Sobolev) mapping spaces between manifolds provide examples of Hilbert manifolds. Then we recall the definition of Jakob's geometric homology and give a discussion of Gysin morphisms both in the finite and in the infinite-dimensional case. 

In section 3, we recall the definition of the Chas--Sullivan product and also of module and coproduct structures, which will be the basic objects of the paper. 

In section 4, we construct first the intersection on the base (in the sense of \ref{ib0}) in the finite-dimensional case and use then naturality and approximation by finite-dimensional manifolds to generalize it to the infinite-dimensional case. After constructing also the intersection in the fiber, we prove several statements about the behavior of the Serre spectral sequence with respect to product, coproduct and module structures. 

In section 5, we begin by considering free loop spaces of sphere bundles. By rational homotopy theory, we can prove that the Serre spectral sequence collapses at $E^2$ in many cases. Then we construct very concrete manifold generators of the free loop spaces of spheres and (complex and quaternionic) projective spaces and use these to prove that the Atiyah--Hirzebruch spectral sequence collapses at $E^2$ for these spaces in various homology theories. At the end, we do a sample computation for the Goresky--Hingston coproduct. 
\subsection*{Acknowledgements}
Most of the results in this paper are part of my diploma thesis \cite{meier2009string}, written under the supervision of Matthias Kreck. To him belongs my gratitude for encouraging support and helpful advice. 

\section{Preliminaries}
\subsection{Hilbert manifolds}
A \textit{Hilbert manifold} is a metrizable space which is locally homeomorphic to a separable Hilbert space $\mathbb{E}$. One can define smooth Hilbert manifolds and their tangents spaces in analogy to the finite-dimensional case. We will assume all Hilbert manifolds to be smooth in this paper. To define later Gysin morphisms, we begin with certain transversality results. We do not claim originality here as similar, but deeper, results were already proven by Quinn in \cite{quinn1970transversal}.

\begin{lemma}\label{translem}Let $E\to M$ be a Euclidean smooth Hilbert space bundle over a compact manifold $M$, possibly with boundary. Furthermore, let $L_0, L_1, \dots \subset E$ be a countable collection of sub Hilbert manifolds of finite codimension and $\varepsilon\co M\to\R$ a positive function on $M$. Then there is a smooth section $s\co M\to E$ with $|s(p)|<\varepsilon(p)$ for all $p\in M$ such that $s$ is transverse to all $L_i$. If $A\subset M$ is a closed submanifold and the zero section is already transverse to the $L_i$ on $A$, we can choose $s|_A = 0$. \end{lemma}
\begin{proof}The proof is completely analogous to the finite-dimensional case if one uses the corresponding results in differential topology for Hilbert manifolds and Kuiper's theorem that every Hilbert space bundle is trivial. See for example my diploma thesis \cite[section 2.4.3]{meier2009string}.\end{proof}

\begin{thm}[Relative Transversality Theorem]\label{transrel}Let $\pi\co E\to B$ be a fiber bundle where the fiber and the base are Hilbert manifolds. Furthermore, let $E^0_0, E^1_0, \dots \subset E$ be a countable collection of subbundles which are in every fiber sub Hilbert manifolds of finite codimension. Let $f\co M\to E$ be a smooth map from a compact manifold $M$. Then there is a homotopy $H\co M\times I\to E$ between $f$ and a map $g\co M\to E$ which is transverse to all $E^i_0$ such that $\pi\circ H = \pi\circ f\circ \pr_1$. If $A\subset M$ is a closed submanifold with $f|_A$ transverse to all $E^i_0$, we can choose $H|_{A\times I} = f\circ\pr_1$.\end{thm}
\begin{remark}
The theorem generalizes immediately to the case where the base is not a Hilbert manifold but a differentiable space in the sense of Sikorski; see the discussion before 2.31 in \cite{meier2009string}. Note also that for $B = pt$, we get simply the usual (absolute) transversality theorem. 
\end{remark}
\begin{proof} Consider the closed embedding $\id\times f\co M \to M\times_B E$, where the pullback is over $\pi\circ f\co M\to B$. We can construct a tubular \nbhd $T$ of $M$ in $M\times_B E$ and identify it with a neighbourhood in the normal bundle. Now choose a section $s\co M\to T$ transverse to all $M\times_B E^i_0$. We can identify the tangent bundle of $M\times_B E$ with $\pr_1^*TM\oplus \pr_2^*T_vE$ where $T_vE$ denotes the vertical part of $TE$. Since $\pi|_{E^i_0}\co E^i_0\to B$ is a submersion for all $i$, we have that $s\co M\to M\times_B E \to M\times E$ is transverse to all $M\times E^i_0$. Hence, we get that $g:= \pr_2\circ s$ ist transverse to all $E^i_0$.\end{proof}

If one uses that every (infinite-dimensional) Hilbert manifold is diffeomorphic to an open subset of the standard Hilbert space, smooth approximation can also be proven just as in the finite-dimensional case. 

An important example for a Hilbert manifold is the space $H^n(M,N)$ of Sobolev maps between a compact manifold $M$ of dimension $n$ and an arbitrary manifold $N$, which is homotopy equivalent to $Map(M,N)$ with the usual compact-open topology. This is surely well known for a long time, but the author was unable to find a complete proof in the literature. A proof can be found in the companion paper \cite{meier2009hilbert}, where also a precise definition of $H^n(M,N)$ and further references are given. In the following, we will write $Map(M,N)$ for $H^n(M,N)$, which will be no source of confusion since at the end we are only interested in the homotopy type.

A useful fact about these mapping spaces is the following approximation theorem:
\begin{thm}\label{map-approx}Let $M,N$ be manifolds and assume $M$ to be compact. Then there exists a sequence of submanifolds $P_1 \subset P_2 \subset \cdots \subset Map(M,N)$ such that one can deform every map $X\to Map(M,N)$ from a compact $X$ to a map into one of the $P_i$. \end{thm} 
This is a generalization of the corresponding well-known theorem for the loop space. A proof along the lines of Milnor's \cite[§16]{milnor1963morse} can be found in \cite[section 2.6.2]{meier2009string}.

\subsection{Geometric homology}\label{gh}
In this section, we want recall a bordism description for homology due to Martin Jakob (\cite{jakob2000alternative}), which works for every (generalized) homology theory. It can be thought as a geometric way to build out of a cohomology theory the corresponding homology theory.

\begin{defi}[Geometric cycles] Let $h^*$ be a cohomology theory and $(X,A)$ a pair of topological spaces. A \textit{geometric cycle} is a triple $(P, a, f)$ where $f \co P \to X$ is a continuous map from a compact connected $h^*$--oriented manifold $P$ with boundary to $X$ such that $f(\partial P)\subset A$ and $a \in h^*(P)$. 

If $P$ is of dimension $p$ and $a \in h^m(P)$, then $(P, a, f)$ is a geometric cycle of \textit{degree} $p - m$.\end{defi}

We want to consider two relations on the class of geometric cycles: 
\begin{enumerate}
\item (Bordism relation) We call two triples $(P,a,f)$ and $(P',a',f')$ bordant if there is a geometric cycle $(W, b, g)$ such that $P\coprod (-P')\subset \partial W$ is a regularly embedded submanifold of codimension 0 which inherits the $h^*$--orientation of $W$. We require further that $b|_P = a, b|_{P'} = a'$,\, $g|_P = f$,\, $g|_{P'} = f'$ and $g(\partial W - P\coprod P')\subset A$. Two bordant cycles are defined to be equivalent.
\item (Vector bundle modification) Let $(P, a, f)$ be a geometric cycle and
consider a smooth $h^*$--oriented $d$--dimensional vector bundle $\pi\co E\to P$, take the unit sphere bundle $S(E \oplus 1)$ of the Whitney sum of $E$ with a copy of the trivial line bundle
over $P$. The bundle $S(E\oplus 1)$ admits a section s. By $s_!\co h^*(P)\to h^{*+d}(S(E\oplus 1))$ we denote the Gysin morphism in cohomology associated to this section\footnote{This can, for example, be defined via Poincare duality. For other possibilites in the analogous case of homology, see the next section.}. We impose that
\[(P, a, f) \sim (S(E \oplus 1), s_!(a), fp).\]\end{enumerate}

We lay upon the group of cycles the equivalence relations generated by the relations 1 and 2. An equivalence class of geometric cycles is called a \textit{geometric class} and will be denoted by $[P, a, f]$. We define $gh_q(X,A)$ to be the abelian group of geometric classes of degree $q$, where addition is defined via disjoint union. 
\begin{thm}[\cite{jakob2000alternative}, Corollary 4.3] There is a natural isomorphism
\begin{eqnarray*}gh_q(X,A) \to h_q(X,A)\end{eqnarray*}
defined by
\begin{eqnarray*}[P, a, f] \mapsto f_*(a \cap [P]),\end{eqnarray*}
where $[P]$ is the fundamental class of $(P,\partial P)$ and $h_*$ is the homology theory corresponding to the spectrum representing $h^*$.\end{thm}

We will identify $gh_*$ and $h_*$ via this isomorphism in the rest of this paper. For later applications, we give an explicit description of the excision isomorphism: Let $[P,a,f]\in h_*(X,A)$ be a geometric class and $B\subset A$ such that $\overline{B}\subset\mathring{A}$. The preimages $f^{-1}(\overline{B})$ and $f^{-1}(X-\mathring{A})$ are closed and we can choose a smooth Urysohn function $g\co P\to \R$ separating them. Choose a regular value $x$ between 0 and 1. Then $Q := g^{-1}([0,x])$ is a manifold with boundary in $A - B$. The restriction $[Q, a|_Q, f|_Q]$ is the image of the excision isomorphism in $h_*(X-B, A-B)$. Indeed, $(P\times [0,1], \pr_1^*(a), f\circ\pr_1)$ is a bordism between $[P,a,f]$ and $i_*[Q, a|_Q, f|_Q]$ since $Q\coprod P$ is a regular submanifold of codimension 0 in $P\coprod P$.\\

\subsection{Gysin morphisms}
For the definition of the Chas--Sullivan product, the construction of Gysin morphisms (also called \textit{umkehr maps} in the literature) is crucial. Let $\iota\co A\hookrightarrow B$ be the inclusion of a sub Hilbert manifold of finite codimension $d$ with $h_*$--oriented normal bundle. We associate to this data a Gysin morphism, i.e.  a ''wrong-way'' map $\iota^!\co h_*(B) \to h_{*-d}(A)$. We will give two constructions in the general case and a third one in the finite-dimensional case.

The first uses the theory of geometric homology\footnote{We follow here (up to sign) Chataur \cite{chataur2005bordism}.}: Let $[P,a,f]$ be a geometric cycle in $h_p(B)$. By the transversality theorem, we can assume that $f$ is transversal to $A$. Now, we define $\iota^!_G([P,a, f]) := [\widetilde{P}, a|_{\widetilde{P}}, f|_{\widetilde{P}}] \in h_{p-d}(A)$, where $\widetilde{P} := f^{-1}(A)$. It is easy to see that this class is well-defined. This is the description in the infinite-dimensional case we will primarly use in this paper. 

Another possible construction uses the Thom isomorphism: Let $\nu\co N \to A$ be a tubular neighborhood of $A$ and $u\in h^d(N, N-A)$ the Thom class. Then the composition 
\[h_*(B) \to h_*(B,B-A)\cong h_*(N, N-A) \xrightarrow{\cap u} h_{*-d}(N) \cong h_{*-d}(A)\]
is an alternative way to define Gysin maps. It coincides with the construction above in the finite-dimensional case and also for mapping spaces between manifolds by the approximation theorem \ref{map-approx}. Therefore, our definition of the Chas--Sullivan product will agree with that of Cohen and Jones \cite{cohen2002homotopy} (which can also be found in Cohen and Voronov \cite{cohen2005notes}) as also shown by Chataur in a different way. Note that one can substitute for $h_*$ here also ordinary homology with \textit{local} coefficients. 

In the finite-dimensional case, there is also a cellular construction of Gysin maps. To make this precise, we need the following (simple) lemma, which is proven in the author's diploma thesis \cite[section 2.5]{meier2009string}:

\begin{lemma}\label{triang}Let $A\subset B$ be a closed submanifold of a finite-dimensional manifold. Then one can triangulate $B$ transversal to $A$ in the sense that every stratum ist transverse to $A$. Furthermore, one can triangulate every $A\cap \Delta_i$ for the simplices $\Delta_i$ of $\mathcal{T}$ in a way such that the triangulations coincide in the intersections $A\cap \Delta_i \cap \Delta_j$. Thus one obtains a induced triangulation of $A$.\end{lemma}

Let $C_*(A)$ and $C_*(B)$ denote the cellular chain complexes. By sending a simplex $\Delta$ of the triangulation of $B$ to the sum of all the simplices of $\Delta\cap A$, we get a chain map $s\co C_*(B)\to C_{*-d}(A)$ which induces a Gysin map on homology. 

This cellular description of the Gysin morphism is also suitable to describe Gysin morphisms for homology with local coefficients: Let $\mathcal{G}$ be a local system and $x = x_{\Delta}$ be the midpoint of a simplex $\Delta$. In the cellular complex with respect to $\mathcal{G}$ the coefficient of $\Delta$ lies in $\mathcal{G}_x$. Choose arbitrary paths from $x$ to the midpoints of the simplices of $\Delta\cap A$ in $\Delta$ and map via them the coefficient of $\Delta$ to coefficients for these simplices (note that all possible choices of these paths are homotopic). This describes a Gysin map $g^!\co H_*(B; \mathcal{G}) \to H_{*-d}(A; \mathcal{G}|_A)$. 

We want to show the equivalence of the cellular construction with the construction via the Thom isomorphism in the case of singular homology. We can form the sub chain complex $T_*(N, N-A)$ of all singular chains $S_*(N, N-A)$ which are transverse to $A$ (i.e. transverse in each stratum). Given a cycle in $S_*(N,N-A)$, one can form an associated simplicial complex $K$ with a map $f\co K\to B$ by glueing the simplices. By a variant of the transversality theorem (see e.g., \cite{meier2009string}, 2.29 and 2.12), one can homotope $f$ to a map $g$ transverse to $A$ without moving $\partial K$ into $A$. Then $(K,g)$ defines a cycle in $T_*(N,N-A)$ which is homologous to $(K,f)$. A similar argument can be applied for boundaries. Therefore, $T_*(N,N-A)$ is quasi-isomorphic to $S_*(N,N-A)$. By the universal coefficient theorem, the dual chain complexes are also quasi-isomorphic. Therefore, it is enough to define the Thom class $u\in H^d(N,N-A)$ on $d$-simplices transverse to $A$. 

For a transverse $d$--simplex $(\Delta, f)$, we define $u(\Delta)$ as the oriented intersection number $\Delta\cap A$. This is a cocycle, since for every transverse $(d+1)$--simplex $(\Delta', f)$ the intersection $\Delta'\cap A$ is represented by a compact $1$--manifold with boundary and the number of oriented boundary points of such a manifold ist $0$. If we restrict $u$ to $(\nu^{-1}(x), \nu^{-1}(x) - \{x\})$ for an $x\in A$, we get the orientation class since $u$ sends the generator of the homology to $1$. Therefore, $u$ represents the Thom class.

The cellular Gysin map is surely unchanged under subdivisions of the triangulation. Therefore, we can assume that every $d$--simplex intersects $A$ in at most one point. Furthermore, we can assume via an isotopy of the triangulation that locally at $A$ every $d$--simplex is only in one fiber. Therefore, we can assume via a suitable subdivision that the $p$--backface of every $(p+d)$--dimensional simplex $\Delta$ intersecting $A$ maps via the projection map homeomorphically onto $\Delta\cap A$. 

A simplex in our triangulation determines a singular simplex with the induced orientation. It is now clear that it is the same if we cap this simplex with $u$ and project it down to $A$ or if we apply the cellular Gysin. This argument clearly works also with (local) coefficients and proves our claim. This shows in particular that the cellular Gysin does not depend on the chosen triangulation at the level of homology.\\

One can also define Gysin morphisms in a relative setting: Let $B'\subset B$ be a submanifold, transverse to $A$. Let furthermore $[P,a,f]\in h_n(B, B')$ be a geometric cycle. Then we can make first $f|_{\partial P}$ transverse to $A\cap B'$ in $B'$ and extend the homotopy to whole $P$ to get a map $f'\co P \to B$. Since this is already transverse to $A$ on $\partial P$, we can homotope $f'$ to a map $g\co P \to B$ which is transverse to $A$ such that the homotopy stays constant on $\partial P$. We define $\iota_![P,a,f] = [\widetilde{P}, a|_{\widetilde{P}}, g|_{\widetilde{P}}]\in h_{n-d}(A, B'\cap A)$, where $\widetilde{P} := g^{-1}(A)$.  

It is also possible to use the Thom map:
\[\xymatrix{h_*(B, B') \ar[d] \\ h_*(B,B'\cup (B-A))\cong h_*(N, \nu^{-1}(A\cap B')\cup (N-A)) \ar[d]^{\cap u}\\ h_{*-d}(N, \nu^{-1}(A\cap B')) \cong h_{*-d}(A, A\cap B')}\]
Here $N$ denotes again a tubular neighborhood of $A$ in $B$. 

\section{The Chas--Sullivan product}\label{cs}
In this section, we want to recall the definition of the Chas--Sullivan product. We will follow the approach by David Chataur exhibited in \cite{chataur2005bordism}. We will fix the notation $L^n M = Map(S^n,M)$ for the unpointed and $\Omega^n M= Map^*(S^n,M)$ for the pointed maps. 

Let $M$ be an $h_*$--oriented manifold of dimension $d$ for a homology theory $h_*$. Consider the diagram

\[ \xymatrix {L^n M\times_M L^n M \ar[d]_{\ev} \ar[r]^\iota & L^n M\times
L^n M\ar[d]^{\ev\times \ev}\\
         M \ar[r]^{\Delta} & M\times M }\]
         
Here $\Delta$ stands for the diagonal, $\iota$ is the inclusion and $\ev$ the evaluation at the base point $pt$ of $S^n$. Since $\ev$ is a submersion, $L^n M\times_M L^n M$ is a sub Hilbert manifold of $L^n M \times L^n M$ and the normal bundle of $\mm$ in $L^n M \times L^n M$ is the pullback of the normal bundle of $M$ in $M\times M$. 

We have a map \begin{eqnarray*}\gamma\co L^n M\times_M L^n M = Map(S^n\vee S^n, M) \to Map(S^n, M) = L^n M\end{eqnarray*} induced by the collapse map $c\co S^n\to S^n\vee S^n$. Note that one has to be careful how to identify the wedge summands with the standard sphere. One possible convention is apparent in the following coordinate description of the collapse map:

\[(x_0,\dots, x_n)\mapsto 
\begin{cases}(2x_0-1, \frac{\sqrt{1-(2x_0-1)^2}}{\sqrt{1-x_0^2}}x_1,\frac{\sqrt{1-(2x_0-1)^2}}{\sqrt{1-x_0^2}}x_2,\dots, \frac{\sqrt{1-(2x_0-1)^2}}{\sqrt{1-x_0^2}}x_n)_1 \text{ for } x_0\geq 0\\
(-2x_0-1, -\frac{\sqrt{1-(2x_0+1)^2}}{\sqrt{1-x_0^2}}x_1, \frac{\sqrt{1-(2x_0+1)^2}}{\sqrt{1-x_0^2}}x_2, \dots,\frac{\sqrt{1-(2x_0+1)^2}}{\sqrt{1-x_0^2}}x_n)_2 \text{ for } x_0\leq 0\end{cases}
\]

The \textit{Chas--Sullivan product} is now defined as the composition

\[ \xymatrix {h_p(L^n M)\tensor h_q(L^n M)\ar[d]^{\times} & h_{p+q-d}(L^n M)\\
h_{p+q}(L^n M \times L^n M) \ar[r]^-{\iota^!} & h_{p+q-d}(L^n M\times_M L^n M) \ar[u]^{\gamma_*} & }\]

For notational convenience, we define $\h_*(L^n M) = h_{*+d}(L^n M)$. Note that, if $h=H$ is ordinary homology, one usually chooses the notation $\mathbb{H}_*(L^n M)$ for $\h_*(L^n M)$. The above composition may now be written as:
\begin{eqnarray*}\mu\co \h_p(L^n M)\tensor \h_q(L^n M) \to \h_{p+q}(L^n M)\end{eqnarray*}

\begin{warning}One has to be careful with signs here since there are different conventions in the literature. The sign convention where the Chas--Sullivan product is graded commutative is used, for example, in the original Chas and Sullivan article \cite{chas1999string} and also in \cite{chataur2005bordism} (where it is ensured by an ''artificial'' sign), while our sign convention agrees, for example, with that in \cite{cohen2005notes} (Theorem 1.2.1 seems to have the wrong sign as stated there): interchanging a factor of degree $p$ and a factor of degree $q$ induces a factor of $(-1)^{pq+d}$.\end{warning}

\begin{remark}As an aside, we remark that it is also possible to give a completely finite-dimensional bordism-like description of the Chas-Sullivan product in the case of ordinary homology via Kreck's theory of stratifolds. Stratifolds are smooth spaces which are stratified by smooth manifolds satisfying certain conditions on the topology and the relationship of the strata. It is possible to carry out much of the usual differential topology in this setting. Especially we can construct a stratifold bordism homology theory, which coincides with singular homology in the case of spaces having the homotopy type of CW-complexes (for more details and precise definitions see Kreck \cite{kreck-differential}). Let now $[S_1,f_1],[S_2,f_2]\in H_*(LM)$ be homology classes. We can interpret a map $f\co S\to LM$ as a map $\ev\circ f\co S\to M$ with a loop in $\Omega(\ev(f(p)))$ attached to each $p\in S$. Intersect $\ev\circ f_1$ and $\ev\circ f_2$ transversally to get a map $F\co S_1\times_M S_2\to M$ and attach to each $p\in S_1\times_M S_2$ the composition of the loops attached to $\pr_1(p)$ and $\pr_2(p)$. This is a representative of the Chas--Sullivan product of $[S_1,f_1]$ and $[S_2,f_2]$ as is shown in the author's diploma thesis \cite[section 3.2.1]{meier2009string}. This is close in spirit to the original definition in Chas and Sullivan \cite{chas1999string}.\end{remark}

A closer look at the definition of the Chas--Sullivan product reveals that the only thing we have used of $S^n$ is that it is both a manifold and an H-cogroup (via the map $S^n \to S^n \vee S^n$). Since every $n$--dimensional manifold $N$ has the structure of an H-comodule over $S^n$ via the map $c\co N \to N\vee S^n$, collapsing the boundary of a little disk, we get (for $M$ $h_*$--oriented) the following module structure:
\[ \xymatrix {h_p(Map(N,M))\tensor h_q(L^n M)\ar[d]^{\times} & h_{p+q-d}(Map(N,M))\\
h_{p+q}(Map(N, M) \times L^n M) \ar[r]^-{\iota^!} & h_{p+q-d}(Map(N,M)\times_M L^n M) \ar[u]^{\gamma_*} }\]

Here $\iota\co Map(N,M)\times_M L^n M \to Map(N,M)\times L^n M$ is the inclusion and \[\gamma\co Map(N,M)\times_M L^n M \to Map(N,M)\] is the map induced by $c$. This defines a $\textbf{h}_*(L^n S)$--module structure on $h_*(Map(N,M))$, which was first considered by Kallel and Salvatore in  \cite{kallel2003rational}. This structure is independent of the chosen disk since all embeddings of a disk are isotopic. Note that we could substitute $h_p(Map(N,M))$ by $h'_p(Map(N,M))$, where $h'$ is a module homology theory over $h$.

Besides the module structure, there is also the structure of a coalgebra on $\textbf{h}_*(LM)$ if $h_*$ is a graded field (e.g., for ordinary homology with field coefficients or Morava K-theory as $h$). To define the coproduct, let $i\co LM \times_M LM \to LM$ be the inclusion of all loops $\alpha\co [0,1] \to M$ with $\alpha(0) = \alpha(\frac12) = \alpha(1)$ and $\iota\co LM \times_M LM\to LM\times LM$ the usual inclusion. Then we get a map
\[ \xymatrix {h_n(LM) \ar[r]^-{i^!} &h_{n-d}(LM \times_M LM) \ar[r]^-{\iota_*} &h_{n-d}(LM\times LM) \cong (h_*(LM)\tensor h_*(LM))_{n-d}. }\]

\label{GH} This coincides as a special case with the TQFT-construction of Cohen and Godin (\cite{cohen2003polarized}). Sadly enough, this coproduct is zero for all classes of degree bigger than $d$, at least in the case of singular homology, by work of Tamanoi (\cite{tamanoi2009loop}) and Goresky--Hingston (\cite{goresky2009loop}). For this reason, Goresky and Hingston constructed (in the cohomological case) an alternative version of the coproduct, which we want to recall here in our language: 
\[\xymatrix{ h_{n-1}(LM, M) \cong h_n(LM\times I, LM \times \partial I \cup M\times I) \ar[d]^J \\
h_n(LM, A) \ar[d]^{i_!} \\
h_{n-d}(Map(8,M), A\cap Map(8,M)) \ar[d]^\cong \\ 
h_{n-d}(Map(8,M), f_1(LM)\cup f_2(LM)) \ar[d]\\ 
h_{n-d}(LM\times LM, LM\times M\cup M\times LM) \ar[d]^\cong \\
h_{n-d}(LM,M) \otimes h_{n-d}(LM, M) }\]
Here, $J\co LM\times I \to LM$ is induced by the map $j\co I\times I \to I$, where $j(t,-)$ sends $0$ to $0$, $\frac12$ to $t$ and $1$ to $1$ and is linear on both halfs of the interval. Furthermore, $A$ is the union of the tubular neighborhoods of the images of the two embeddings $f_{1/2}\co LM \to LM$, sending a loop to the same loop with doubled speed on one half of the interval and constant on the other. This $A$ can be chosen as a subbundle since $f_{1/2}(LM)$ are subbundles of $LM \to M$. 

By the same way, we can define a coproduct on $h_*(\Omega M, pt)$. We call both of these coproducts \textit{Goresky--Hingston coproducts} and denote them by $\Psi_{GH}$. Intuitively, they can be seen as unparametrized versions of the Cohen--Godin coproduct.  

\section{Spectral sequences}\label{ss}
\subsection{Preliminaries}\label{sss}
Let $\xi = (E\stackrel{\pi}{\to} B)$ be a fiber bundle with $B$ a path-connected CW-complex and fiber $F$ and $\xi' = (E'\xrightarrow{\pi'} B)$ be a subbundle with fiber $F'$. We define $E^{(p)}$ to be the preimage of the $p$--skeleton of $B$ under $\pi$ and $E'^{(p)}$ correspondingly. Recall that for a homology theory $h_*$ the defining exact couple $C(\xi, \xi')$ of the associated Serre spectral sequence $\EE(\xi, \xi')$ is given by 
\begin{eqnarray*}\xymatrix{ \bigoplus\limits_{p,q} h_{p+q}(E^{(p)}, E'^{(p)})\ar[rr]^i& &\bigoplus\limits_{p,q} h_{p+q}(E^{(p)}, E'^{(p)})\ar[dl]^j\\
& \bigoplus\limits_{p,q} h_{p+q}(E^{(p)}, E^{(p-1)}\cup E'^{(p)})\ar[ul]^k } \end{eqnarray*}
Here the morphisms are the same as in the exact sequence of the triple \[(E^{(p)}, E^{(p-1)}\cup E'^{(p)}, E'^{(p)})\] via the excision isomorphism $h_*(E^{(p-1)}\cup E'^{(p)}, E'^{(p)}) \cong h_*(E^{(p-1)}, E'^{(p-1)})$. So, two of the three maps are induced by inclusions, the third is given by restricting a cycle $[P,a,f]$ to a codimension $0$ submanifold of the boundary. The $E^1$--term is isomorphic to \[\bigoplus_{p-\mbox{cells } \alpha \mbox{ of } B}\hspace{-0.3cm}[(h_*(D^p_\alpha, S^{p-1}_\alpha)\otimes_{h_*(pt)} h_*(\pi^{-1}(x_\alpha), \pi'^{-1}(x_\alpha))]_{p+q},\] with $x_\alpha \in int(\alpha)$, i.e. the $p$-th part of the cellular complex computing $H_p(B; \mathfrak{h}_q(F, F'))$ where $\mathfrak{h}_*(F, F')$ denotes the local system given by the homologies of the fibers. It is easy to work out that, if we have a smooth fiber bundle of smooth manifolds, the isomorphism sends a geometric cycle $z = [P,a,f]$ to \begin{eqnarray*}z \cap \pi^{-1}(x_\alpha) := [P\cap\pi^{-1}(x_\alpha), a|_{P\cap\pi_{-1}(x_\alpha)}, f|_{P\cap\pi_{-1}(x_\alpha)}]\end{eqnarray*} for $f \pitchfork \pi^{-1}(x_\alpha)$. Here one uses the explicit description of the excision morphism given in \ref{gh}. Note furthermore that if $x$ and $y$ are two regular values of $\pi f$ and $\gamma\co I \to B$ is a path with $\gamma(0) = x$ and $\gamma(1) = y$ which is also transverse to $\pi f$, then the fiber transport of the homology class $z\cap \pi^{-1}(x)$ along $\gamma$ is $z\cap \pi^{-1}(x)$.

To ease the formulation of the results of the next sections, we want to fix some general terminology for spectral sequences. A \textit{morphism of spectral sequences} $E^*_{**}$ and $\tilde{E}^*_{**}$ of level $k$ and bidegree $(a,b)$ consists of homomorphisms $f^n\co E^n_{pq}\to \tilde{E}^n_{p+a,q+b}$ for all $n\geq k$ which commute with the differentials and satisfy $H(f^n) = f^{n+1}$. Now assume that $E$ and $\tilde{E}$ converge to graded abelian groups $D^\infty_*$ and $\tilde{D}^\infty_*$ which are filtered by $F^*_*$ and $\tilde{F}^*_*$ respectively. If we have a morphism $f^*_{**}$ between $E$ and $\tilde{E}$ and in addition homomorphisms $D^{\infty}_r \to \tilde{D}^{\infty}_{r+a+b}$ which map $F^p_r$ to $\tilde{F}^{p+a}_{r+a+b}$ and induce $f^\infty$ on $E^{\infty}$, we speak of a \textit{morphism of convergent spectral sequences}. Morphisms of convergent exact couples induce morphisms of the associated convergent spectral sequences. 

\subsection{Intersecting on the base and in the fiber}\label{ifb}
The goal of this subsection is to define Gysin morphisms of Serre spectral sequences which ''compute'' the corresponding Gysin morphisms of the homology of the total space. 

\subsubsection{Intersecting on the base}
Let $\xi = (F\to E\stackrel{\pi}{\to} B)$ be a fiber bundle with $B$ a (finite-dimensional) manifold and $A\subset B$ a closed submanifold of codimension $d$ with $h_*$--oriented normal bundle. Choose a triangulation of $B$ transverse to $A$ and triangulate $A$ as in \ref{triang}. 

\begin{thm}[Intersecting on the base]\label{ib} There is a morphism $s_B(A)$ of convergent spectral sequences of level 1 and bidegree $(-d,0)$ between $\EE(\xi)$ and $\EE(\xi|_A)$ where the spectral sequences are defined by the triangulations above. The morphism is canonical starting with level 2 and induces the usual Gysin morphism $H_p(B; \mathfrak{h}_q(F))\to H_{p-d}(A; \mathfrak{h}_q(F))$ on this level.\end{thm}
\begin{proof}We want to construct a morphism of the corresponding exact couples $C(\xi)$ and $C(\xi|_A)$. That means, we need to construct morphisms \[\sigma_a\co h_{p+q}(E^{(p)})\to h_{p+q-d}(\pi^{-1}(A^{(p-d)}))\] and \[\sigma_r\co h_{p+q}(E^{(p)},E^{(p-1)})\to h_{p+q-d}(\pi^{-1}(A^{(p-d)}),\pi^{-1}(A^{(p-1-d)}))\] which commute with the boundary maps. We concentrate on the relative case since this is more difficult. 
 
Let $[P,a,f]\in h_{p+q}(E^{(p)}, E^{(p-1)})$. We want to find a homotopy \begin{eqnarray*}(P,\partial P)\times I \to (E^{(p)}, E^{(p-1)})\end{eqnarray*} from $f$ to a map $g$ such that $\pi g$ is transverse to $A$. To smooth $\pi f$, consider open neighbourhoods $U_p$ of the $p$--skeleton $B^p$ with $U_{p-1}\subset U_p$ such that there are smooth retracts $r_p\co U_p\to B^p$ with $r_p|_{U_{p-1}} = r_{p-1}$. In addition, we can assume that $A\cap U_p$ is mapped to $A\cap B^p$ by $r_p$. There is a homotopy $H_1\co \partial P\times I \to U_{p-1}$ from $\pi f|_{\partial P}$ to a smooth map. Extend this homotopy to a homotopy $H_1\co P \times I \to U_p$ from $\pi f$ to a map $\tilde{f}$. This map $\tilde{f}$ is smooth on $\partial P$, so we can find a homotopy $H_2\co P\times I\to U_p$ to a smooth map such that $H_2|_{\partial P} = \tilde{f}\circ\pr_1$. Now, $r_p\circ H_2(x,1)$ is homotopic to $\pi f$ and smooth. Since $E\to B$ is a fibration, we can lift this homotopy. Therefore, we can assume $\pi f$ to be smooth.  

We can homotope $\pi f|_{\partial P}$ in $B^{(p-1)}$ to be transverse to $A\cap B^{(p-1)}$ since we can first do this in $U_{p-1}$ and then use $r_p$. We can extend this to a map $\tilde{f}$ on the whole of $P$. Since $\tilde{f}$ is tranverse to $A$ on $\partial P$, we can homotope it to a $\tilde{g}\co P\to B^{(p)}$ in $B^{(p)}$ which is transverse to $A\cap B^{(p)}$ while leaving $\partial P$ fixed. We can lift this homotopy to $E$ and get a map $g\co P\to E^{(p)}$, for which $\pi g = \tilde{g}$ is transverse to $A$. 

Since $B^{p}\cap A \subset A^{p-d}$, we can define the maps $\sigma_a$ and $\sigma_r$ by transverse intersection of our representative $P$ with $A$. More precisely, define $Q:= (\pi f)^{-1}(A)$ and send $[P,a,f]$ to \[[Q, a|_Q, f|_Q] \in h_{p+q-d}(E|_{A^{(p-d)}}, E|_{A^{(p-1-d)}})\] and the same in the absolute case. Since $\partial Q = \partial P\cap Q$, we get a morphism of exact couples and therefore of convergent spectral sequences $E^1_{\ast\ast}(\xi) \to E^1_{(\ast-d)\ast}(\xi|_A)$.

We now have to check that it induces the usual Gysin morphism on $E^2$. To that end, choose the $x_\alpha$ for the cells of $B$ to be regular values of $\pi f$ and those for the cells of $A$ to be regular values of $\pi f|_Q$.  By choosing paths transverse to $\pi f$ from the $x_\alpha$ for cells $\alpha$ of $B$ to the $x_\beta$ for $\beta$ in the intersection of $\alpha$ with $A$, one sees that our construction coincides with the cellular description of the Gysin morphism by the previous subsection.\end{proof}

\begin{prop}[Naturality]\label{nat}Let $\phi\co E'\to E$ be a map of fiber bundles $\xi'= (F'\to E'\to B')$ and $\xi = (F\to E\to B)$. Let $A\subset B$ be a submanifold and the map on the bases $f\co B'\to B$ be transverse to $A$. Then the following diagram commutes beginning with the second level:

\[ \xymatrix{ \EE(\xi')\ar[r]^{\phi_*}\ar[d]^{s_B(f^{-1}(A))} & \EE(\xi) \ar[d]^{s_B(A)}\\
\EE(\xi'|_{f^{-1}(A)}) \ar[r]^-{\phi_*} & \EE(\xi|_A) } \]
\end{prop}
\begin{proof}The normal bundle of $f^{-1}(A)$ in $B'$ is the pullback of the normal bundle of $A$ in $B$. Since the Thom class is natural, the proposition follows via the Thom isomorphism description of the Gysin morphism.\end{proof}

We now want to generalize the intersection morphism to an infinite-dimen\-sional context. So let $\xi = (F\to E\to B)$  be a fiber bundle with projection map $\pi$ where $B$ is a Hilbert manifold and $A\subset B$ a closed sub Hilbert manifold of codimension $d$ with $h_*$--oriented normal bundle. Assume, furthermore, that there is a collection of finite-dimensional manifolds $P_1\subset P_2\subset\cdots\subset B$ with inclusions $\iota^i_{i+j}\co P_i\to P_{i+j}$ and $\iota_i\co P_i\hookrightarrow B$ such that every map $f\co X\to B$ from a compact space can be homotoped into one of the $P_i$ -- this is, for example, the case if $B$ is a mapping space (see \ref{map-approx}).

\begin{prop}\label{ib2}In the situation above, there is a (canonical) morphism $s_B(A)$ of convergent spectral sequences of level 2 and bidegree $(-d,0)$ between $\EE(\xi)$ and $\EE(\xi|_A)$. The morphism induces the usual Gysin morphism $H_p(B; \mathfrak{h}_q(F))\to H_{p-d}(A; \mathfrak{h}_q(F))$ on $E^2$. \end{prop}
\begin{proof}Let $x$ be in $E^n_{pq}(\xi)$. This element is represented by an element $z$ in $E^2_{pq}\cong H_p(B; \mathfrak{h}_q(F))$ with $d_2(z) = d_3(z) = \cdots = d_{n-1}(z) = 0$. Since 
\[H_*(B; \mathfrak{h}_*(F)) \cong \colim_i  H_*(P_i; \mathfrak{h}_*(F)),\]
 there is an $N\in\N$ such that there is a preimage $y\in E^2_{pq}(\xi|_{P_N})$ of $z$ under the map \[(\iota_N)_*\co \EE(\xi|_{P_N})\to \EE(\xi)\] with $d_2(y) = d_3(y) = \cdots = d_{n-1}(y) = 0$, therefore representing a preimage $[y]$ of $x$ in $E^n_{pq}(\xi_{P_N})$. By the transversality theorem, we can assume that $P_N$ is transverse to $A$, hence $P_N\cap A \subset P_N$ is a closed submanifold of $P_N$. We now define $s_B(A)(x) = (\iota_N)_*s_B(P_N\cap A)([y])$. We have to check that this is a well-defined map and that it defines a morphism of spectral sequences.

The map is independent of the choices because of the naturality of intersecting on the base: Suppose $(\iota_{n_1})_*(y_1) - (\iota_{n_2})_*(y_2) = d_k(u)$ is a boundary for some $k<n$ where $y_i\in E^2(\xi|_{P_{n_i}})$, $i=1,2$, are cycles. We can find a $v\in E^2(\xi|_{P_{N}})$ ($N$>>0) with $\iota_{N}(v)$ representing $u$ and $\iota^{n_1}_N(y_1)- \iota^{n_2}_N(y_2) = d_k(v)$. Now we use that $\iota_{n_i}$ factors over $\iota_N$, that intersecting on the base is natural and that the $\iota^{n_1}_N[y_1] = \iota^{n_2}_N[y_2]$ in $E^n(\xi|_{P_N})$ to deduce that our map is well-defined.

The map is a morphism of spectral sequences since intersecting on the base is a morphism of spectral sequences for each $P_i$. 
\end{proof}

\begin{remark}If $\xi$ is a smooth fiber bundle of Hilbert manifolds, it is clear by the construction of the intersection on the base that it converges to the usual Gysin morphism on the homology of the total spaces. Here we take in the case of twisted coefficients the Thom isomorphism description of the Gysin map, which makes perfect sense in the twisted setting.\end{remark}

\subsubsection{Intersecting in the fiber}
Let $\xi = (F\to E\stackrel{\pi}{\to} B)$ be a smooth fiber bundle with fiber a Hilbert manifold and base a finite-dimensional manifold and $\xi' = (F'\to E'\xrightarrow{\pi'}B)$ be an open subbundle. Let $\xi_0 = (F_0\to E_0\to B)$ be a subbundle of constant codimension $d$ and $h_*$--oriented normal bundle. Denote by $\xi'_0 = (F'_0\to E'_0 \to B)$ the intersection of the two subbundles. Note that one can often reduce from other situations to the open subbundle case by a tubular neighborhood argument. 

\begin{thm}[Intersecting in the fiber]\label{if1} There is a morphism $s_F(E_0)$ of convergent spectral sequences of level 1 and bidegree $(0,-d)$ between $\EE(\xi, \xi')$ and $\EE(\xi_0, \xi'_0)$. This induces the usual Gysin morphism $H_p(B; \mathfrak{h}_q(F, F'))\to H_p(B; \mathfrak{h}_{q-d}(F_0, F'_0)))$ on $E^2$. Furthermore, it converges to the Gysin morphism in the homology of the total spaces.\end{thm}
\begin{proof} We want to define a morphism of the corresponding exact couples which induces the usual Gysin morphism on the $E^2$-term. We will only discuss explicitely the case of the $E$-term, the others are similar. Let \[[P,a,f]\in h_{p+q}(\pi^{-1}(B^p), \pi^{-1}(B^{p-1})\cup\pi'^{-1}(B^p))\] be a homology class. We need to find a homotopy from $f$ to a $g$ such that $g \pitchfork E_0$, $g|_{\partial P} \pitchfork E_0$, $g \pitchfork \pi^{-1}(x_\alpha)$ (for a point $x_\alpha$ in the interior of every $p$-cell) and $g|_{f^{-1}(\pi^{-1}(x_\alpha))} \pitchfork E_0\cap \pi^{-1}(x_\alpha)$. The first is necessary to define the Gysin morphism, the second to insure that this is compatible with the morphisms in the exact couple and the last two guarantee that it coincides with the Gysin morphism on $E^2$. As in the proof of \ref{ib}, we can as a first thing assume that $f$ is smooth. 

Choose the $x_\alpha$ to be regular values in every $p$-cell $\alpha$ of $\pi f\co (\pi f)^{-1}(B^p-B^{p-1}) \to B^p-B^{p-1}$ and of the restriction $\pi f|_{\partial P}$. Then $f$ and $f|_{\partial P}$ are transverse to the $\pi^{-1}(x_\alpha)$. Choose disks $D_\alpha$ inside $B^p-B^{p-1}$ around the $x_\alpha$ and trivialize $\xi$ on them as $D_\alpha \times F$ and $\xi'$ as $D_\alpha \times F'$. After pulling the bundle $\xi$ back along a smooth self map $B^p \to B^p$ homotopic to the identity, which is identity outside the $D_\alpha$ and maps a small disk around the $x_\alpha$ constantly to $x_\alpha$, we can, after possibly making $D_\alpha$ smaller, assume that $\xi'$ embeds into $\xi$ on $D_\alpha$ as $D_\alpha \times F' \subset D_\alpha \times F$. We have $(\pi f)^{-1}(D_\alpha) \cong D^p\times (\pi f)^{-1}(x_\alpha)$ and we can by a homotopy assume the function $\pr_2 f\co D^p\times (\pi f)^{-1}(x_\alpha) \to F$ to be constant on every 
$D^p\times \{y\}$. Now use the transversality theorem to make first $f|_{(\pi f)^{-1}(x_\alpha)\cap \partial P}$ transverse to $\pi'^{-1}(x_\alpha) \cap E_0 \subset \pi'^{-1}(x_\alpha)$ and then $f|_{(\pi f)^{-1}(x_\alpha)}$ transverse to $\pi^{-1}(x_\alpha) \cap E_0\subset \pi^{-1}(x_\alpha)$ for all $p$-cells $\alpha$ and extend this homotopy ''constantly'' on $D^p\times (\pi f)^{-1}(x_\alpha)$ and on the rest in an arbitrary way. After possibly making the disks smaller, we know that $f$ is already transverse to $E_0$ on the $D^p\times (\pi f)^{-1}(x_\alpha)$. View now $f$ as a map $P \to \pi^{-1}(U_p)$ for $U_p$ as in the last subsection. Then we can make first $f|_{\partial P}$ and then $f$ transverse to $E_0$ by a homotopy leaving $f$ fixed on the $D^p\times (\pi f)^{-1}(x_\alpha)$ such that the image of $f$ is still contained in $\pi^{-1}(B^p)$ by theorem \ref{transrel}. 

By intersecting with $E_0$, we now get a morphism of the exact couples which commutes with the boundary operator as above. More precisely we map $[P,a,f]$ to $[Q, a|_Q, f|_Q]\in h_{p+q-d}(E^{(p)}\cap E_0, (E^{(p-1)}\cup E'^{(p)})\cap E_0)$ with $Q = f^{-1}(E_0)$. This induces a convergent morphism $\EE(\xi, \xi')\to \EE(\xi_0, \xi'_0)$ of level 1 and bidegree $(0,-d)$. That this map induces the Gysin morphism on $E^2$ can be seen by the explicit isomorphism of the $E^1$--term to the cellular complex: there is no difference if we intersect first with $F$ and then with $F_0$ or if we first intersect with $E_0$ and then with $F_0$ (if everything is transverse). The well-definedness is proven as usual. \end{proof}

As in the case of the intersection on the base, we can extend the cases we are interested in to an infinite-dimensional context. So let now $B$ be a Hilbert manifold and the other notation as above and assume that there is a collection of finite-dimensional manifolds $P_1\subset P_2\subset\cdots\subset B$ such that every map $f\co X\to B$ from a compact space can be homotoped into one of the $P_i$. We will state only the absolute form of the theorem since it is enough for our applications. 

\begin{thm}\label{if} In the situation above, there is a morphism $s_F(E_0)$ of convergent spectral sequences of level 2 and bidegree $(0,-d)$ between $\EE(\xi)$ and $\EE(\xi_0)$. This induces the usual Gysin morphism $H_p(B; \mathfrak{h}_q(F))\to H_p(B; \mathfrak{h}_{q-d}(F_0)))$ on $E^2$. Furthermore, it converges to the Gysin morphism in the homology of the total spaces.\end{thm}
\begin{proof}As in the case of the intersection on the base.\end{proof}

\subsection{Multiplicative, comultiplicative and module structures}\label{sm} 
Define $E[(a,b)]$ to be the shifted spectral sequence with $E^k[(a,b)]_{pq} = E^k_{(p+a)(q+b)}$. For a $d$--dimensional manifold $M$, set $\HH_*(M) = H_{*+d}(M)$ and also recall the notation $\textbf{h}_*(M) = h_{*+d}(M)$.

\begin{thm}\label{cohen2003loop}Let $M$ be a $d$--dimensional $h_*$--oriented manifold. Then 
\begin{eqnarray*}\EE(\Omega^n M \to L^n M\to M)[(d,0)]\end{eqnarray*}
 can be equipped with the structure of a multiplicative spectral sequence which converges to the Chas--Sullivan product on $\textbf{h}_*(L^n M)$. Furthermore, the induced product on the $E^2$--term $\HH_*(M; \mathfrak{h}_*(\Omega^n M))$ is equal to the intersection product with coefficients in the local system of rings $\mathfrak{h}_*(\Omega^n M)$ whose multiplication is given by the Pontryagin product.\end{thm}
\begin{proof}Let $\xi = (\Omega^n M \to L^n M\to M)$ and denote by $\Delta_M\co M\to M\times M$ the diagonal. Then we define the multiplicative structure as the composition
\begin{eqnarray*}\EE(\xi)\tensor \EE(\xi) \xrightarrow{\times} \EE(\xi\times\xi) \xrightarrow{s_B(\Delta_M)} \EE(\Delta_M^*(\xi\times\xi)) \xrightarrow{\gamma_*} \EE(\xi).\end{eqnarray*}
Here $\gamma$ is defined as in section \ref{cs}. Note that the cross product is a map of spectral sequences. All claims follow from the corresponding ones for Gysin morphisms.\end{proof}

In the next theorem, we need the Chas-Sullivan product in ordinary homology with local coefficients. While in section \ref{cs} we have only considered untwisted coefficients, we can simply use the Thom isomorphism construction for the Gysin map as above.

\begin{thm}\label{fss}Let $M\to N\to O$ be a fiber bundle of $h_*$--oriented manifolds of dimensions $m$, $n$ and $o$ respectively. Then 
\begin{eqnarray*}\EE(L^n M\to L^n N \to L^n O)[(o,m)]\end{eqnarray*}
 can be equipped with the structure of a multiplicative spectral sequence which converges to the Chas--Sullivan product on $\textbf{h}_*(L^n N)$. Furthermore, the induced product on the $E^2$--term $\HH_*(L^n O; \mathfrak{h}_{*+m}(L^n M))$ is equal to the Chas--Sullivan product with coefficients in the local system of rings $\mathfrak{h}_{*+m}(L^n M)$.\end{thm}
\begin{proof}Let $\xi = (L^n M\to L^n N\to L^n O)$, $\nu = (L^n M\times_M L^n M\to L^n N\times_N L^n N \to L^n O\times_O L^n O)$ and $\iota\co L^n O\times_O L^n O \hookrightarrow L^n O \times L^n O$ be the inclusion. We define the multiplicative structure as the composition
\[ 
\EE(\xi)\tensor \EE(\xi)\stackrel{\times}{\to}\EE(\xi\times\xi)\xrightarrow{s_B(\iota)}
\EE(\iota^*(\xi\times\xi)) \xrightarrow{s_F(L^n N\times_N L^n N)}
\EE(\nu) \xrightarrow{\gamma_*}
\EE(\xi)
\]
Here $\gamma$ is again defined as in section \ref{cs}. By proposition \ref{map-approx}, we are in the situation of theorem \ref{ib2} and the intersection morphism is defined.
\end{proof}

There is also the notion of a \textit{comultiplicative spectral sequence}, which is simply a comonoid in spectral sequences, i.e. one has a map $E\to E\otimes E$ which is coassociative (we will not consider counits).

In section \ref{cs} we have defined the Goresky--Hingston coproducts on $h_*(LM)$ and on $h_*(\Omega M)$ for $h_*$ being a graded field. For $V$ a coalgebra over a field $k$, we have furthermore a coproduct on $H_*(M; V)$:
\begin{eqnarray*}H_*(M; V) \to H_*(M; V\tensor V)\xrightarrow{\Delta_*} H_*(M\times M; V\tensor V)\cong H_*(M;V)\otimes_k H_*(M; V)\end{eqnarray*} The same diagram also defines a coproduct in the case of local coefficients. 

\begin{thm}Let $M$ be a $d$--dimensional $h_*$--oriented manifold with $h_*$ a graded field. Then \begin{eqnarray*}\EE((\Omega M, pt) \to (LM, M)\to M)[(0,d-1)]\end{eqnarray*} can be equipped with the structure of a comultiplicative spectral sequence which converges to the Goresky--Hingston coproduct on $h_{*+d-1}(LM, M)$. Furthermore, the induced coproduct on the $E^2$--term $H_*(M; \mathfrak{h}_{*+(d-1)}(\Omega M, pt))$ is equal to the coproduct on $M$ with coefficients in the Goresky--Hingston coalgebra local system $\mathfrak{h}_{*+(d-1)}(\Omega M, pt)$.\end{thm}
\begin{proof}Let $\xi = \Omega M \to LM \to M$ and view $(I, \partial I)$ as a bundle over the point and $M$ as the identity bundle $M\to M$. Furthermore let $A$ be as in the end of section \ref{GH} and $\xi_{1/2}$ defined by the embeddings $f_{1/2}$. Consider the diagramm
\[ \xymatrix{
\EE(\xi, M)\cong \EE((\xi, M)\times (I,\partial I)) \ar[d]^{J} \\
\EE(\xi, A)\ar[d]^{s_F(LM\times_M LM)} \\
\EE(\xi\times_M\xi,A\cap \xi \times_M\xi)\cong \EE(\xi\times_M\xi,\xi_1 \cup \xi_2) 
\ar[d]^{i}\\
\EE((\xi,M)\times(\xi,M))\xrightarrow{\cong}
\EE(\xi, M)\otimes \EE(\xi, M)
}\]

All claims follow from the corresponding ones for Gysin morphisms.\end{proof}

There are also spectral sequence results for the coproduct defined by Cohen and Godin, which we will not state here explicitely.

For $E$ a multiplicative spectral sequence, there is also the notion of a module spectral sequence, i.e. a spectral sequence $E'$ together with a morphism $E\otimes E'\to E'$ and the usual coherence diagrams. Recall that, if $M$ is a $d$--manifold and $N$ a module over a ring $R$, we have an $\HH_*(M; R)$--module structure on $\HH_*(M; N)$ defined analogous to the intersection product (here we use the cross product $H_*(M; R)\otimes H_*(M; N)\to H_*(M\times M; N)$). 

\begin{thm}Let $Z$ be a closed $n$--manifold and $M$ be a $d$--dimensional $h_*$--oriented manifold. Then \begin{eqnarray*}\EE(Map^*(Z, M)\to Map(Z,M)\to M)[(d,0)]\end{eqnarray*} can be equipped with the structure of a module spectral sequence over $\EE(\Omega^n M\to L^n M\to M)[(d,0)]$ which converges to the module structure on $\h_*(Map(Z,M))$. Here $Map^*$ stands for the pointed maps. Furthermore, the induced module structure on the $E^2$--term $\HH_*(M; \mathfrak{h}_*(Map^*(Z, M))$ coincides with the module structure described above.\end{thm}
\begin{proof}As in the multiplicative case.\end{proof}

\begin{thm}Let $Z$ be a closed $n$--manifold. Furthermore, let $M\to N\to O$ be a fiber bundle of $h_*$--oriented manifolds of dimensions $m$, $n$ and $o$ respectively. Then \begin{eqnarray*}\EE(Map(Z, M)\to Map(Z, N)\to Map(Z, O))[(o,m)]\end{eqnarray*} can be equipped with the structure of a module spectral sequence over $\EE(L^n M\to L^n N\to L^n O)[(o,m)]$ which converges to the module structure on $\h_*(Map(Z,N))$. Furthermore the induced module structure on the $E^2$--term \begin{eqnarray*}H_{*+o}(Map(Z,O); \mathfrak{h}_{*+m}(Map(Z, M))\end{eqnarray*} coincides with the module structure described in section \ref{cs}.\end{thm}
\begin{proof}As in the multiplicative case.\end{proof}

\section{Examples}
In this subsection, we will do three different things, in object and method. First, we want to widen our knowledge about ordinary homology of free loop spaces to the case of certain sphere bundles. Secondly, we want to compute some extraordinary homologies of free loop spaces, for example Landweber exact theories (e.g., complex cobordism and complex K-Theory), Morava K-theory and oriented bordism. We will study the Atiyah--Hirzebruch spectral sequence associated to spheres and (complex) projective space and show that it degenerates on $E^2$. To achieve this, we need to construct first explicit manifold generators for the ordinary homology of the free loop spaces of the spheres and projective spaces, which may be interesting in its own right. At last, we turn our attention to the Goresky--Hingston coproduct. 

\subsection{The case of sphere bundles}\label{esb}
We want to study the homology of the free loop space of certain sphere bundles. While the integral homology of free loop spaces is usually hard to compute, there is a more efficient tool for the rational homology, namely rational homotopy theory. Recall that rational homotopy theory associates functorially to every simply-connected space $X$ a (graded) commutative differential graded algebra (dga) over $\Q$ of the form $\Lambda V$. Here $V$ is a graded rational vector space and $\Lambda V$ is the free (graded) commutative dga over $V$, i.e. a tensor product of polynomial rings for the basis elements of $V$ of even degree and exterior algebras for the odd parts. This is called the \textit{minimal model} $\mathcal{M}(X)$ of $X$. The cohomology of $\mathcal{M}(X) = \Lambda V$ is isomorphic to the rational cohomology of $X$ (see F{\'e}lix, Halperin and Thomas \cite{félix2001rational}). 

In addition, we will need the following two facts of rational homotopy theory:

\begin{enumerate} \item The vector space $V$ is naturally isomorphic to the dual of $\pi_*(X;\Q) := \pi_*(X)\otimes \Q$ (see \cite{félix2001rational}, Thm 15.11).
\item The minimal model of $LX$ depends only on the minimal model of $X$. This can be seen by the explicit formulas of Vigu{\'e}-Poirrier and Sullivan in \cite{vigué1976homology}.\end{enumerate}

While the minimal model of $LX$ only gives information about the rational cohomology, we want to use rational homotopy theory in combination with the Serre spectral sequence to do integral computations for the free loop space $LE$ of a fiber bundle $S^k\to E\to S^n$. We use the computation of $\HH(S^n)$ for $n>1$ by Cohen, Jones and Yan \cite{cohen2003loop}:
\begin{eqnarray*}\mathbb{H}_*(LS^n) &=& \Lambda(a) \otimes \mathbb{Z}[u] \text{ for }n\text{ odd,}\\
 \mathbb{H}_*(LS^n) &=& \Lambda(b)\otimes \mathbb{Z}[a,v]/(a^2,ab,2av) \text{ for }n\text{ even}\end{eqnarray*}

with generators $a\in\mathbb{H}_{-n}(LS^n)$, $b\in\mathbb{H}_{-1}(LS^n)$, $u\in\mathbb{H}_{n-1}(LS^n)$ and $v\in\mathbb{H}_{2n-2}(LS^n)$. For this computation, they used the multiplicative spectral sequence exhibited in \ref{sm} in the special case of singular homology. 

First assume $k>1, n>1$ odd. The odd dimensional spheres have only one nontrivial rational homotopy group, namely $\pi_k(S^k;\Q) = \Q$. Hence $\pi_i(E;\Q) = \pi_i(S^k;\Q)\oplus \pi_i(S^n;\Q)$ in every degree by the long exact sequence of homotopy groups. So we have $\mathcal{M}(E) = \Lambda(x_k)\tensor\Lambda(x_n)$ with $|x_k| = k$ and $|x_n| = n$. For dimension reasons, there are no differentials. Thus we have $\mathcal{M}(E) \cong \mathcal{M}(S^k\times S^n)$ as differential graded algebras. We conclude \[H_*(LE; \Q) \cong H_*(L(S^k\times S^n);\Q).\] Consider the $E^2$--term of the Serre spectral sequence associated to $LS^k\to LE\to LS^n$. Every occuring group is torsionfree. Therefore, our rational computation shows that the spectral sequence degenerates at $E^2$ and we have \[\HH_*(LE) \cong \HH_*(LS^k)\tensor \HH_*(LS^n) \cong \Lambda(a_k,a_n)\tensor \Z[u_k, u_n]\] with $|a_k| = -k$, $|a_n| = -n$, $|u_k| = k-1$ and $|u_n| = n-1$ (for the grading conventions, see section \ref{cs}). Note that all extension are trivial in the sense that 
\begin{eqnarray*}0\to F^{n-1}\to F^n\to F^n/F^{n-1}\to 0\end{eqnarray*}
splits since all occuring groups in $E^\infty$ are torsionfree. To show that the isomorphism above holds also multiplicatively, we apply the following proposition to the Serre spectral sequence associated to $LM\to LN\to LO$: 

\begin{prop}\label{mult}Let $E$ be a multiplicative convergent first quadrant spectral sequence of modules over a noetherian ring $R$ that converges multiplicatively to a graded group $G_*$ and has the grading conventions of the homological Serre spectral sequence. Assume that $E^\infty_{**} \cong E^\infty_{*0}\tensor_R E^\infty_{0*}$ and that this is finitely generated over $R$ in every bidegree. Furthermore, require that all filtration extensions are trivial and $E^\infty_{*0} = R[\overline{x}_0, \overline{x}_1,\dots \overline{x_n}]\tensor \Lambda_R(\overline{x}_{n+1},\overline{x}_{n+2},\dots)$, $|\overline{x_i}|$ odd for $i>n$. Then $G_k \cong \bigoplus_{p+q = k}E_{pq}^\infty$ holds multiplicatively.\end{prop}
\begin{proof}Denote the filtration of $G_q$ by $F^*_q$. As $E^\infty_{p0} = F_p^p/F_p^{p-1}$ and $F_p^p = G_p$, we have a surjective map $G_p\to E^\infty_{p0} = E^2_{p0}$. Lift the $\overline{x_j}$ to $x_j$ in $G_p$. Since the multiplication on $E^\infty = F^p/F^{p-1}$ is induced by that on $G_*$, we have that $x_i x_j$ is a lift for $\overline{x_i}\overline{x_j}$. The groups $E^\infty_{0*}$ act on $G_*$ and $E^\infty$ in a compatible way via multiplication. Consider the map 
\[L\co E^\infty_{**}\to G_*,\hspace{0.4cm}y\cdot \Pi \overline{x_i}^{k_i}\mapsto y\cdot \Pi x_i^{k_i},\]
 where $y\in E^\infty_{0*}$. This map is clearly a map of algebras. It is also clear that it is surjective onto $F^0$. Assume inductively that it is surjective onto $F^p$. The products $\Pi\overline{x_i}^{k_i}$ with $\sum k_i|x_i| = p+1$ form a $E^\infty_{0*}$--basis for $F^{p+1}/F^p$ and are images of $L$. Therefore, we see that $L$ is surjective onto $F^{p+1}$ and conclude by induction that it is surjective onto the whole of $G_*$. Since $E^\infty_{**} \cong G_*$ additively and both are finitely generated (hence noetherian) $R$--modules in every degree, $L$ is an isomorphism (of algebras). Indeed, identify $E^\infty_{**}$ and $G_*$ and consider the ascending chain $\ker(L^n)$, $n\in\mathbb{N}$. For some $N$, we have $\ker(L^N) = \ker(L^{N+1})$. Since every $y\in G_*$ is of the form $L^N(y')$, we get $Ly = 0$ iff $y=0$.\end{proof}

Now we consider the case $k>1$ odd and $n>2$ even. Assume furthermore that $k\neq n\pm 1$ and that $n-1$ is no multiple of $k-1$. Even dimensional spheres $S^n$ have two non-zero rational homotopy groups, namely $\pi_n(S^n;\Q)\cong\Q$ and $\pi_{2n-1}(S^n,\Q)\cong \Q$. By the long exact homotopy sequence, we have \[\pi_i(E;\Q) = \pi_i(S^k;\Q)\oplus\pi_i(S^n;\Q)\] in every degree. So we get \[\mathcal{M}(E) \cong \Lambda(x_k)\tensor \Z[x_n]\tensor\Lambda(y_{2n-1})\] with $|x_k| = k$, $|x_n| = n$ and $|y_{2n-1}| = 2n-1$. Since the Serre spectral sequence associated to $S^k\to E\to S^n$ degenerates at $E^2$, we have $d(x_k) = d(x_n) = 0$ and $d(y_{2n-1})$ must be a non-zero multiple of $x_n^2$. Therefore, $\mathcal{M}(E)$ is isomorphic to the minimal model $\mathcal{M}(S^k\times S^n)$ and hence we have \[H_*(LE;\Q)\cong H_*(LS^k\times LS^n; \Q).\] Consider the Serre spectral sequence associated to $LS^k\to LE\to LS^n$. A differential $d_i(x)$ can only be non-zero if $d_i(x)$ is torsion. The only torsion elements of $\HH_*(LS^n)$ are the $av^j$ for $j\geq 1$ (see \ref{es} for notation).  Hence, we have \[d_i(1\tensor a_{LS^k}) = d_i(1\tensor u_{LS^k}) = d_i(a_{LS^n}\tensor 1) = d_i(b_{LS^n}\tensor 1) = d_i(v_{LS^n}\tensor 1) = 0\] for all $i\geq 2$ as one sees by an analysis of possible differentials. The $E^2$--term of the Serre spectral sequence is isomorphic to $\HH_*(LS^n)\tensor\HH_*(LS^k)$. By multiplicativity, the spectral sequence degenerates at $E^2$. Because filtration issues may come up, we cannot deduce in this case the concrete structure of the homology. 

\subsection{Manifold generators}
In this subsection, we will exhibit concrete manifold generator for the homology classes of the free loop spaces of spheres and projective spaces. Our basic source for the computation of these homology rings is again Cohen, Jones and Yan \cite{cohen2003loop}.
\subsubsection{The spheres}\label{es}
In this section, we will present concrete generators for the homology of $LS^n$ for $n>1$. To achieve this, we consider first the simpler case of $\Omega S^n$. It is well known that $H_*(\Omega S^n) \cong \mathbb{Z}[x]$ with $x\in H_{n-1}(\Omega S^n)$, where the product is induced by composing loops. By adjunction from the identity, we get a map $f\co S^{n-1} \rightarrow \Omega\Sigma S^{n-1} \cong \Omega S^n$. This represents a class in $H_{n-1}(\Omega S^n)$, which is easy to be seen a additive generator.

To visualize $x$, think of the base point as the north pole. The points $p\in S^{n-1}$ of the equator parametrize the minimal geodesics $\gamma_p$ between north and south pole. Now choose a distinguished minimal geodesic $\delta$ from the south to the north pole (the ''way backwards''). Then $p \mapsto \delta\ast\gamma_p$ defines $f\co S^{n-1} \to \Omega S^n$ where $\ast$ denotes the concatenation of paths (note that the suspensions above are reduced).

Now consider the free loop space. We use the same notation for the generators of homology as in the last section. Since $a$ is in $H_0(LM)$, it can be represented by an arbitrary loop, for example a constant loop. By studying the Serre spectral sequence for $\Omega M \to LM\to M$, it can be shown that $j_*\co H_{n-1}(\Omega M) \to \mathbb{H}_{-1}(LM)$ is an isomorphism (see \cite{cohen2003loop}). Therefore, $j_*(x)$ is a generator of $\mathbb{H}_{-1}(LM)$ and hence up to sign equal to $b$ for $n$ even and to $au$ for $n$ odd. 

For identifying the other generators, we begin with the easier case of $n$ \textit{odd}. Consider the Gysin morphism \[j^!\co H_*(LS^n)\to H_{*-d}(\Omega S^n).\] We want to find a preimage of $x$ under $j^!$. This is a fortiori a generator of $\HH_{n-1}(LS^n)$ and therefore up to sign equal to $u$. Let $S^n$ be equipped with the standard metric of the sphere of circumference 1 and $STS^n$ be the unit sphere bundle in the tangent bundle $TS^n$. Let $V$ be a vector field of unit length. We define a map $F\co STS^n\to LS^n$ by

\begin{eqnarray*}(p,v)\mapsto \left(t\mapsto \begin{cases}
\exp_p(tv) \text{ for } t\leq \frac12 \\
\exp_{-p}((t-\frac12)V(-p)) \text{ for } t\geq\frac12\end{cases}\right)\end{eqnarray*}

Here $p$ denotes a point in $S^n$ and $v$ is a unit tangent vector to $p$. By the description of $x$ above it is clear that $j^!(F_*[STS^n]) = x$.

This construction cannot work for $n$ even since in this case we have $\HH_{n-1}(LS^n) = 0$ for $n>2$ and the generator $bv\in \HH_{n-1}(LS^2)$ maps to zero under $j^!$, because a representative $(S,f)$ can be chosen with $\im(\ev\circ f) = pt$. \footnote{This gives an eccentric proof for the theorem of the hairy ball, because the only thing we used for $n$ odd was the existence of a non-vanishing vector field. }

To construct an explicit representative of $v$, we need an alternative representative of $x^2$. By our description above we get as a representative:

\begin{eqnarray*}(v_1,v_2)\mapsto \left(s\mapsto\begin{cases}
\exp_p(2sv_1) \for s\leq \frac14\\
\exp_{-p}(2(s-\frac14)w) \for \frac14\leq s\leq\frac12\\
\exp_p(2(s-\frac12)v_2) \for \frac12\leq s \leq \frac34\\
\exp_{-p}(2(s-\frac34)(-w)) \for \frac34\leq s \leq 1\end{cases}\right)\end{eqnarray*}

Here $v_1$ and $v_2$ are unit tangent vectors at the base point $p$ and $w$ is a unit tangent vector at $-p$. This map is now easy to be seen to be homotopic to

\begin{eqnarray}\label{Formel}(v_1,v_2)\mapsto \left(s\mapsto\begin{cases}
\exp_p(sv_1) \for s\leq \frac12\\
\exp_{-p}((s-\frac12)(-v_2)) \for \frac12\leq s \leq 1\\
\end{cases}\right)\end{eqnarray}

where $-v_2$ denotes the parallel transport of $v_2$ along any geodesic from $p$ to $-p$. By this description, it is now easy to construct a preimage of $x^2$ under $j^!$: consider the pullback $E$ of the product bundle $STS^n\times STS^n$ via the diagonal $S^n \to S^n\times S^n$. Then construct $F\co E\to LS^n$ by defining $F(p,v_1,v_2)$ by the formula above (\ref{Formel}). This is obviously a preimage of $x^2$. Since $x^2$ is an additive generator of $H_{2n-2}(\Omega S^n)$ and $\Z\{v\}$ is the non-torsion part of $\HH_{2n-2}$, the class $[E,F]$ equals $v$ up to sign. 

By the Chas--Sullivan product, we get now explicit generators for every class in $H_*(LM)$. In particular every class in $H_*(LM)$ is represented by a manifold. 

\begin{remark}We could also have used the explicit generators to deduce the multiplicative structure of $H_*(LM)$ from the additive structure, without using the compatibility of the Chas--Sullivan product with the Serre spectral sequence. \end{remark}

\subsubsection{Projective spaces}\label{eps}
In \cite{cohen2003loop} the authors show that 
\begin{eqnarray*}\mathbb{H}_*(L\mathbb{C}\mathbb{P}^n) \cong \Lambda(w)\otimes\mathbb{Z}[c,u]/(c^{n+1}, wc^n, (n+1)c^nu),\end{eqnarray*}
where $|w| = -1$, $|c|=-2$ and $|u| = 2n$. By the same methods, one can show
\begin{eqnarray*}\mathbb{H}_*(L\mathbb{H}\mathbb{P}^n) \cong \Lambda(w)\otimes\mathbb{Z}[c,u]/(c^{n+1}, wc^n, (n+1)c^nu)\end{eqnarray*} 
with $|w| = -1$, $|c|=-4$ and $|u|=4n+2$. We will find explicit generators for these homology classes. In the following, $\mathbb{K}$ will stand for one of $\mathbb{C}$ or $\mathbb{H}$ and $d$ will be the $\mathbb{R}$-dimension of $\mathbb{K}$. 

\label{plp}
There is a homotopy fibration $S^{d(n+1)-1}\rightarrow \mathbb{KP}^n\rightarrow \mathbb{KP}^\infty$ (see \cite{cohen2003loop}). If we loop this, we get a homotopy section $\Omega(\kp^\infty)\to \Omega (\kp^n)$, since the map $\Omega\kp^\infty \simeq S^{d-1} \to S^{d(n+1)-1}$ is nullhomotopic. Therefore, we get (additively): 
\begin{eqnarray*}H_*(\Omega\mathbb{KP}^n) \cong H_*(\Omega S^{d(n+1)-1})\otimes H_*(S^{d-1})\cong \Z[y]\otimes \Lambda(z),\end{eqnarray*}
where $|y| = d(n+1)-2$ and $|z| = d-1$. Here we use $\Omega \mathbb{KP}^\infty \simeq S^{d-1}$. The above isomorphism holds also multiplicatively, because the Serre spectral sequence is here multiplicative, since the Pontrjagin product is induced by a map and the Serre spectral sequence is natural (note that there are no filtration issues for dimension reasons). 

In this whole section, let the projective spaces be equipped with the metric coming from the standard metric of the
sphere of circumference 1. The generator $c$ is represented by $c\co \mathbb{KP}^{n-1}\hookrightarrow\mathbb{KP}^n\hookrightarrow L\mathbb{KP}^n$. This holds because $\ev\circ c\co \mathbb{KP}^{n-1}\rightarrow \mathbb{KP}^n$ represents a generator of $\HH_{-d}(\kp^n)$.

Let $d>1$ and $STS^{d(n+1)-1}$ be the unit sphere bundle of the tangent bundle of $S^{d(n+1)-1}$. Let $E$ be the quotient of the sphere bundle via the (isometric) action of $S^{d-1}$ on $S^{d(n+1)-1}$. Clearly, the composite $STS^{d(n+1)-1}\rightarrow LS^{d(n+1)-1} \rightarrow L\mathbb{KP}^n$ of the generator $u$ of $H_*(LS^{d(n+1)-1})$ and the looped Hopf map factors over $E$. This map $E\rightarrow L\mathbb{KP}^n$ is our generator $u$. To see this, we simply intersect $u$ with $\Omega\mathbb{KP}^n$ and observe that this equals $y$ in the notation above.

For the last generator look at $T\mathbb{KP}^n|_{\mathbb{KP}^{n-1}}$. This has certainly a nonvanishing section $s$ by obstruction theory, because the $(dn)$-th cohomology of $\mathbb{KP}^{n-1}$ vanishes. We can assume that $s(x)$ has unit length for every $x\in\mathbb{KP}^{n-1}$. Let now $L$ be the (trivial) $S^{d-1}$-bundle in $T\mathbb{KP}^n|_{\mathbb{KP}^{n-1}}$ generated by $s$. Then we define our map $w'\co L\rightarrow L\mathbb{KP}^n$ via

\begin{eqnarray*}(p,l) \mapsto \begin{cases} t\mapsto \exp_p(\frac12 t\cdot l)\mbox{ for } t\leq\frac12\\
t\mapsto \exp_p(\frac12 (1-t)\cdot s(p)) \mbox{ for } t \geq \frac12 \end{cases}\end{eqnarray*} 

If we multiply the represented homology class $[w']$ with $c^{n-1}$, i.e. intersect with a $\kp^1$ connecting a $p\in \kp^{n-1}$ with $\exp_p(\frac14 s(p))$, we get obviously the image of the generator $[f]\in H_{d-1}(\Omega S^d)\cong H_{d-1}(\Omega\kp^{1})$ under the map $\Omega\kp^{1}\hookrightarrow\Omega\kp^n\hookrightarrow L\kp^n$. The first is an isomorphism on $H_{d-1}$ because of the description of the homology of $\Omega\kp^n$ above. The latter is also an isomorphism on $H_1$ by inspection of the Serre spectral sequence associated to $\Omega\kp^n\to L\kp^n\to \kp^n$ (see \cite{cohen2003loop}). So $w'$ is an additive (non-torsion) generator since $[w']c^{n-1}$ is. Since $\mathbb{H}_{-1}(L\kp^n) \cong \mathbb{Z}$, this settles $w'$ (modulo sign) as the generator $w$ described in \cite{cohen2003loop}. 

\begin{remark}A similar description in the case $\mathbb{K} = \mathbb{R}$ can be found in the author's diploma thesis \cite[section 3.6.2]{meier2009string}.\end{remark}

\subsection{Calculations in generalized homology theories}
\subsubsection{Complex cobordism and Landweber-exact theories}\label{cc}
We define natural transformations $\mu\co MSO_n(X) \to H_n(X)$ and $\nu\co MU_n(X) \to H_n(X)$ by sending a representing cycle $[M,f]$ to $f_*([M])$. 

We have the following well-known proposition:
\begin{prop}If $X$ is (homotopy equivalent to) a CW-complex, then its Atiyah--Hirzebruch spectral sequence (AHSS) for $MU$ degenerates at $E^2$ if and only if $\nu\co MU_n(X)\to H_n(X)$ is an epimorphism for all $n\geq 0$.\end{prop}

To see the degeneration of the AHSS for $MU$, we have to find stably almost complex structures on our generators of \ref{es} and \ref{eps}. The sphere $S^n$ is framed and therefore stably almost complex. Recall, we denoted the sphere subbundle of the tangent bundle by $STS^{n}$. We have that $STS^n\subset S(TS^n\oplus\epsilon) \cong S^n\times S^n$ has trivial normal bundle. Since the tangent bundle of $S^n\times S^n$ is stably trivial, the tangent bundle of $STS^n$ is stably trivial, too, and therefore stably almost complex. This finishes the case for the sphere.

The manifolds $\cp^n$ and $\cp^n\times S^1$ are clearly almost complex. It remains to show that $STS^{2n+1}/S^1$ is stably almost complex, where $S^1$ acts via complex multiplication and its derivative. As in the paragraph above, it suffices to consider $(S^{2n+1}\times S^{2n+1})/S^1$, where $S^1$ acts via the diagonal action. Embed $S^{2n+1}\times S^{2n+1}$ into $\C^{n+1}\times\C^{n+1}$. This gives an embedding $(S^{2n+1}\times S^{2n+1})/S^1 \hookrightarrow (\C^{n+1}\times\C^{n+1})/\C^*\cong \cp^{2n+1}$ of codimension 1. Since the latter is complex and the normal bundle is trivial (note that $(S^{2n+1}\times S^{2n+1})/S^1$ is simply connected), degeneration is proven. 

By localizing at $p$, we can deduce that the AHSS for $LM$ with $M$ a sphere or a complex projective space degenerates at $E^2$ also for $BP$. It is not difficult to show that one gets thereby degeneration also for all Landweber-exact theories, i.e. homology theories $h_*$ of the form $h_*(X) = MU_*(X)\otimes_{MU_*} h_*$ or $h_*(X) = BP_*(X)\otimes_{BP_*} h_*$. This includes, among others, complex K-homology, elliptic homology, the Johnson--Wilson theories $E(n)$ and the Morava $E$--theories $E_n$. Note that the isomorphism $h_*(LM) = MU_*(LM)\otimes_{MU_*} h_*$ or $h_*(LM) = BP_*(LM)\otimes_{BP_*} h_*$ holds multiplicatively with respect to the Chas--Sullivan product for an $MU_*$-- respectively $BP_*$--oriented manifold $M$. 

A further class of homology theories for which we get degeneration is provided by all homology theories $h$ with $\Tor^1(H_*(LM), h_*) = 0$ and a natural transformation $MU \to h$ which is surjective on coefficients. Indeed, we get in this case a map of Atiyah-Hirzebruch spectral sequences $\EE(LM, MU) \to \EE(LM, h)$, which is surjective on every $E^r$-term. Furthermore, we get also degeneration for all homology theories $h'$ with $h'_*(X) \cong h_*(X) \otimes_{h_*} h'_*$ for all $X$. Examples for such homology theories include the Morava K-theories $k(n)$ and $K(n)$ and the spectra $P(n)$ and $B(n)$ for a fixed prime $p$ where $H_*(LM)$ has no $p$-torsion. The transformations $MU \to h$ are in all these examples multiplicative (see, for example, \cite{EKMM}, V.4). Therefore, the Chas--Sullivan product is in these examples determined by that in the $MU$-case. 

For the odd dimensional spheres, we have in addition that all filtration extension are trivial, i.e.
\begin{eqnarray*}0\to F^{n-1}\to F^n\to F^n/F^{n-1}\to 0\end{eqnarray*}
splits since $F^n/F^{n-1}$ is a free $h_*$--module. Therefore, we have additively
\begin{eqnarray*}h_*(LS^{2k+1}) \cong H_*(LS^{2k+1})\tensor h_*\end{eqnarray*}
for all mentioned homology theories $h_*$. 

We have $E^2(LM, MU)_{pq} = H_p(LM)\otimes MU_*(pt)$. By \ref{mult}, the above isomorphisms
\begin{eqnarray*}MU_*(LS^{2k+1}) &\cong& H_*(LS^{2k+1})\tensor MU_*\\
BP_*(LS^{2k+1}) &\cong& H_*(LS^{2k+1})\tensor BP_*\end{eqnarray*}
hold now also multiplicatively. Thereby, we can conclude an analogous isomorphism for all mentioned homology theories. 

In the light of the physical interest in equivariant index theory on the free loop space, it would be exciting to extend these calculations to the $S^1$-equivariant K-theory of free loop spaces. 

\subsubsection{Oriented Bordism}
In analogy to \ref{cc}, Conner and Floyd show:
\begin{prop}[\cite{conner1979differentiable}, 15.1]If $X$ is (homotopy equivalent to) a CW-complex, then the Atiyah--Hirzebruch spectral sequence for oriented bordism degenerates at $E^2$ if and only if $\mu\co MSO_n(X)\to H_n(X)$ is surjective for all $n\geq 0$.\end{prop}

As we have described concrete manifold generators in \ref{es} and \ref{eps}, we get degeneration for free loop spaces of spheres and (complex and quaternionic) projective spaces. But we can prove even more in some cases:

\begin{thm}[\cite{conner1979differentiable}, 15.2]If $X$ is (homotopy equivalent to) a CW-complex for which each $H_n(X)$ ist finitely generated and has no odd torsion, then \begin{eqnarray*}MSO_n(X)\cong \bigoplus_{p+q=n} H_p(X; MSO_q)\end{eqnarray*}\end{thm}

We can apply this theorem to these free loop spaces of spheres and complex or quaternionic projective spaces which have no odd torsion in homology. Therefore, we have additive isomorphisms
\begin{eqnarray*}MSO_*(LS^{2k}) &\cong& H_*(LS^{2k}; MSO_*(pt))\\
MSO_*(LS^{2k+1}) &\cong& H_*(LS^{2k+1})\otimes MSO_*(pt)\\
MSO_*(\cp^{2^k-1}) &\cong& H_*(L\cp^{2^k-1}; MSO_*(pt))\\
MSO_*(\hp^{2^k-1}) &\cong& H_*(L\hp^{2^k-1}; MSO_*(pt))
\end{eqnarray*}

Sadly enough, $MSO_*$ is not torsionfree, but has also $2$--torsion (a complete determination can be found in \cite{wall1960determination}). Therefore, we can deduce only that these isomorphisms hold also for the multiplicative structure if we localize at an odd prime, except in the case of odd-dimensional spheres, where the above isomorphism holds multiplicatively for all primes. 

\subsection{Computation of Coproducts}
The aim of this section is to use explicit generators for the homology of $\Omega S^n$ to compute the Goresky-Hingston loop coproduct for the based loop space and deduce via a Serre spectral sequence argument from this the corresponding coproduct structure on the homology of $LS^n$. More general computations by other methods with somewhat weaker results were already done in \cite{goresky2009loop}, 13.9.

We have to replace the generators for $H_*(\Omega S^n) \cong \mathbb{Z}[x]$ in \ref{es} by a certain pertubation to ensure transversality. More precisely, we choose a point $P$ on the equator which does not lie on the ''way backwards'' $\delta$ as the new base point and conjugate the generator presented in \ref{es} by the shortest geodesic from $P$ to the north pole. We want to compute the coproduct of $x^k$, represented by the map $(S^{n-1})^k \to \Omega S^n$ which is the $k$-th power of the map just described. Precomposing with $J$, we get a map $(S^{n-1})^k \times I \to \Omega S^n$, where we call the second variable the \textit{height}. Outside the boundary, this intersects $\Omega S^n \times \Omega S^n$ transversally at the heights $\frac{l+\frac12}{k+\frac12}$ for $l= 0,\dots, k-1$ with level sets $(S^{n-1})^l\times (S^{n-1})^{k-l-1} \to \Omega S^n\times \Omega S^n$, representing $x^l\times x^{k-l-1}$. Therefore, the Goresky--Hingston coproduct on $H_*(\Omega S^n, \{P\})$ (which is here well-defined even for integer coefficients) is given by the formula 
\[\Psi_{GH}(x^k) = \sum_{l=1}^{k-2} x^l \otimes x^{k-l-1}. \]

The coproduct on $\HH_*(S^n) \cong \Lambda(a)$ is given by $\Delta(1) = 1\otimes a + a \otimes 1$ and $\Delta(a) = a\otimes a$. The coproduct structure on the $E^2$-term of the Serre spectral sequence for $\Omega S^n \to LS^n \to S^n$ is just the tensor product of the two coproduct structures. The $E^\infty$-term is isomorphic to $E^2$ for $n$ odd and carries the subquotient coproduct structure for $n$ even for every field coefficients. Filtration issues do not occur, so this gives the coproduct structure for $LS^n$. Note that the $E^\infty$-term does \textit{not} inherit a coproduct structure for $n$ even if we choose \textit{integer} coefficients. 

\bibliographystyle{gtart}
\bibliography{StringTopology}
\end{document}